\documentclass[reqno]{amsart}
\usepackage{mathptmx,pxfonts,amssymb,amsfonts,amsmath,graphicx,shadow,color}
\usepackage[all]{xy}
\usepackage{mathrsfs,tikz,tikz-cd}
\usetikzlibrary{arrows}
\usepackage{hyperref}
\usepackage[english]{babel}
\usepackage[latin1]{inputenc}
\usepackage[T1]{fontenc}
\usepackage{lineno}
\newtheorem{thm}{Theorem}[section]
\newtheorem{corollary}[thm]{Corollary}
\newtheorem{lemma}[thm]{Lemma}
\newtheorem{proposition}[thm]{Proposition}
\newtheorem*{properties}{Properties}

\theoremstyle{remark}
\newtheorem{remark}[thm]{Remark}

\DeclareMathOperator{\Ext}{Ext}

\DeclareMathOperator{\depth}{depth}

\DeclareMathOperator{\Spec}{Spec}

\DeclareMathOperator{\Tor}{Tor}

\newcommand{\codim}{\operatorname{codim}}
\newcommand{\codepth}{\operatorname{codepth}}
\newcommand{\dd}{\operatorname{d}}
\newcommand{\idd}{\operatorname{id}}
\newcommand{\embdim}{\operatorname{embdim}}

\begin{document}

\title[Tensor product of algebras over a ring]{Cohen-Macaulay, Gorenstein, Complete intersection and regular defect for the tensor product of algebras}

\author[S. Bouchiba]{S. Bouchiba $^{(\star)}$}
\address{Department of Mathematics, Faculty of Sciences, University Moulay Ismail, Meknes, Morocco}
\email{bouchibasamir@gmail.com}

\author[J. Conde-Lago]{J. Conde-Lago}
\address{Departamento de \'Alxebra, Facultade de Matem\'aticas, Universidade de Santiago de Compostela, E-15782 Santiago de Compostela, Spain}
\email{jesus.conde@usc.es}

\thanks{$^{(\star)}$ Supported by KFUPM under DSR Research Grant \# RG1212.}

\author[J. Majadas]{J. Majadas}
\address{Departamento de \'Algebra, Facultad de Matem\'aticas, Universidad de Santiago de Compostela, E-15782 Santiago de Compostela, Spain}
\email{j.majadas@usc.es}

\date{\today}


\begin{abstract}
This paper main goal is to measure the defect of Cohen-Macaulayness,
Gorensteiness, complete intersection and regularity for the tensor
product of algebras over a ring. For this sake, we determine the
homological invariants which are inherent to these notions, such as
the Krull dimension, depth, injective dimension, type and embedding
dimension of the tensor product constructions in terms of those of
their components. Our results allow to generalize various theorems
in this topic especially \cite[Theorem 2.1]{BK1}, \cite[Theorem 6]{TY}
and \cite[Theorems 1 and 2]{M}
as well as two Grothendieck's theorems on the transfer of Cohen-Macaulayness
and regularity to tensor products over a field
issued from finite field extensions. To prove our theorems on the
defect of complete intersection and regularity, the homology theory
introduced by Andr\'e and Quillen for commutative rings turns out to
be an adequate and efficient tool in this respect.
\end{abstract}
\maketitle


\section{Introduction}

All rings of this paper are assumed to be unitary and Noetherian and
all ring homomorphisms are unital. In particular, the tensor product
$A\otimes_RB$ of two Noetherian algebras $A$ and $B$ over a
Noetherian ring $R$ is supposed to be Noetherian.

In \cite{BK1}, the authors were concerned with the problem of the
transfer of the Cohen-Macaulayness property to the tensor product of
algebras over a field. Its main theorem proves that the tensor
product $A\otimes_kB$ of two algebras $A$ and $B$ over a field $k$,
assuming Noetherianity of $A\otimes_kB$, is Cohen-Macaulay if and
only if so are $A$ and $B$. This theorem extends a result of
Grothendieck in this regard when the components of the considered
tensor product are two field extensions one of which is finitely
generated over $k$. The transfer of the Gorenstein and complete intersection
properties to these tensor product constructions was 
investigated by Watanabe, Ishikawa, Tachibana and Otsuka in \cite{WITO}
and subsequently by Tousi and Yassemi in \cite{TY}, proving similar
theorems for these two notions as the above one on Cohen-Macaulayness.

The regularity property still resists to any effort of improvement
at the level enjoyed by the above three notions. Indeed, contrary to
the above notions, a Noetherian tensor product of two field
extensions of a field $k$ is not regular in general. In 1965,
Grothendieck proved that $K\otimes_kL$ is a regular ring provided
$K$ or $L$ is a finitely generated separable extension field of $k$
\cite[Seconde Partie, Lemma 6.7.4.1]{EGA4} and the finiteness hypotheses
were dropped in \cite[Note]{Sh1}. In 1969, Watanabe, Ishikawa, Tachibana,
and Otsuka, showed that under a suitable condition tensor products
of regular rings are complete intersections \cite[Theorem 2, p.
417]{WITO} and subsequently Tousi and Yassemi proved that a
Noetherian tensor product of two $k$-algebras $A$ and $B$ is regular
if and only if so are $A$ and $B$ in the special case where $k$ is
perfect. In this context, the main theorem of \cite{BK2} establishes
necessary and sufficient conditions for a Noetherian tensor product
of two extension fields of $k$ to inherit regularity and hence
generalizes Grothendieck's aforementioned result. Recently,
a more general result (over a base ring) was obtained in \cite{M}.

In the stream of the above achievements, we aim in this paper to go
farther in this direction by measuring the defect of
Cohen-Macaulayness, Gorensteiness, complete intersection and
regularity for the tensor product of algebras over a ring $R$. To
this purpose, it is clearly essential to determine the very
invariants which are inherent to these notions such as
the Krull dimension, depth, injective dimension, type and embedding
dimension of the tensor product constructions in terms of those of
their components. That is our objective in the most part of this
paper which allows us to evaluate the defect of the above notions
for the tensor products under some flatness condition. Our results
allow to generalize the above theorems in \cite{BK1}, \cite{BK2}, \cite{M},
\cite{TY}, \cite{WITO}. In the last two sections of the paper 
we use the homology theory of Andr\'e and Quillen.
For the convenience of the reader, we include a whole
section to give a brief introduction to such homology, including the results
used in this paper.\\

\noindent {\bf Acknowledgement}\\
We are indebted to S. Kabbaj for his many contributions to this paper.

\section{Cohen-Macaulay rings}
This section main goal is to determine the Cohen-Macaulayness defect
of the tensor product $A\otimes_RB$ of algebras $A$ and $B$ over an
arbitrary ring $R$ in local and global cases. Also, we generalize
\cite[Theorem 2.1]{BK1} which characterizes the Cohen-Macaulayness of
the tensor product $A\otimes_kB$ of algebras $A$ and $B$ over a
field $k$ in terms of the Cohen-Macaulayness of $A$ and $B$.

Let $R$ be a ring and let $A$ and $B$ be $R$-algebras. First, it is
worth noting that the tensor product $A\otimes_RB$ over $R$ might be
trivial even if $A$ and $B$ are not so. Of course, the interesting
case is when $A\otimes_RB\neq \{0\}$ which makes it legitimate to
seek conditions on the underlying components $R$, $A$ and $B$ of the
tensor product which guarantee $A\otimes_RB\neq \{0\}$. That is the
purpose of our first result.

For a given $R$-algebra $A$, denote by $f_A:R\longrightarrow A$,
with $f_A(r)=r\cdot1_A$ for any $r\in R$ and where $1_A$ is the unit
element of $A$, the ring homomorphism defining the structure of
algebra of $A$ over $R$. Also, if $A$ and $B$ are $R$-algebras, we
denote by $\mu_A:A\longrightarrow A\otimes_RB$ and
$\mu_B:B\longrightarrow A\otimes_RB$ the canonical algebra
homomorphisms over $A$ and $B$, respectively, such that
$\mu_A(a)=a\otimes_R1$ and $\mu_B(b)=1\otimes_Rb$ for each $a\in A$
and each $b\in B$. Observe that the following diagram is
commutative:

$$\begin{array}{ccccc}
&&A&&\\
&\stackrel {f_A}\nearrow &&\stackrel {\mu_A}\searrow&\\
R&&\stackrel {f_{A\otimes_RB}}\longrightarrow&&A\otimes_RB\\
&\stackrel {f_B}\searrow&&\stackrel {\mu_B}\nearrow&\\
&&B&&
\end{array}$$

\noindent Throughout this section, given ideals $I$, $J$ and $H$ of
$A$, $B$ and $A\otimes_RB$, respectively, we adopt the following
notation for easiness: $I\cap R:=f_A^{-1}(I)$, $J\cap
R:=f_B^{-1}(J)$ and $H\cap A:=\mu_A^{-1}(H)$, $H\cap
B:=\mu_B^{-1}(H)$. By commutativity of the above diagram
$(H\cap A)\cap R = (H\cap B)\cap R$.
Also, given an ideal $J$ of
$B$, we denote by $A\otimes_RJ$ the ideal
$\mu_B(J)(A\otimes_RB)$ of $A\otimes_RB$ generated by
$\mu_B(J):=\{1\otimes_Rb:b\in J\}$.

We begin by characterizing when the tensor product $A\otimes_RB$ of
two algebras $A$ and $B$ over a ring $R$ is not trivial. Also, we
characterize when, given two prime ideals $I$ and $J$ of $A$ and
$B$, respectively, there exists a prime ideal $P$ of $A\otimes_RB$
such that $P\cap A=I$ and $P\cap B=J$.

\begin{proposition}\label{2.1}
Let $R$ be a ring and $A,B$ be two $R$-algebras. Then

(1) $A\otimes_RB\neq \{0\}$ if and only if there exists a prime ideal
$I$ of $A$ and a prime ideal $J$ of $B$ such that $I\cap R=J\cap
R$.

(2) Let $I$ be a prime ideal of $A$ and $J$ be a prime ideal of $B$. There exists a prime ideal $P$ of $A\otimes_RB$ such that $P\cap
A=I$ and $P\cap B=J$ if and only if $I\cap R=J\cap R$.
\end{proposition}
\begin{proof}
(2) See \cite[Corollaire 3.2.7.1.(i)]{EGA1}.\\
(1) Since $A\otimes_RB\neq \{0\}$ if and only if $\Spec(A\otimes_RB)\neq\emptyset$, it follows from (2).
\end{proof}

We denote by $\mathcal R$ the class of Noetherian local rings
$(S,m,k)$, where $m$ is the maximal ideal of $S$ and $k$ its residue
field. Let $\Lambda$ consist of the functions
$$\lambda: \; \mathcal R \;\longrightarrow\; \mathbb{N}\cup\{+\infty\}$$
vanishing on any field and satisfying following property: for any flat local homomorphism
$(S,m,k)\longrightarrow (T,n,l)$, we have
$\lambda(T)=\lambda(S)+\lambda\Big (\displaystyle {\frac T{mT}}\Big
)$. Since from the identity map of a field $k$ we obtain $\lambda(k)=\lambda(k)+\lambda(k)$, the condition that $\lambda$ vanishes on field $k$ is equivalent to $\lambda(k)\neq+\infty$. The sum of two such functions is also an element of $\Lambda$ as well as the difference (provided it is defined and takes values in $\mathbb{N}\cup\{+\infty\}$).

The following invariants are well known examples of elements of
$\Lambda$.

\begin{proposition}\label{2.2} The following functions from $\mathcal R$ to
$\mathbb{N}\cup \{+\infty\}$ are elements of $\Lambda$:

(1) The Krull dimension (denoted by $\dim$).

(2) The depth (denoted by $\depth$).

(3) The self-injective dimension (denoted by $\idd$).

(4) The codepth (denoted by $\codepth$).

(5) The complete intersection defect (denoted by $\dd$).
\end{proposition}

\begin{proof} (1) See \cite[Theorem 15.1]{Mt}.\\
(2) See \cite[Corollary, p. 181]{Mt} or \cite[Seconde Partie, Proposition 6.3.1]{EGA4}.\\
(3) See \cite[Corollary 1]{FT}.\\
(4) Since $\codepth=\dim-\depth$, then (4) follows from (1) and (2).\\
(5) See \cite[Proposition 3.6]{Av}.
\end{proof}

 Let $\Lambda'$ be the class of functions
 $\lambda':\mathcal R\longrightarrow \mathbb{N}\cup\{+\infty\}$
 verifying
 $$\lambda'(R_n)+\lambda'((A\otimes_RB)_P) \;=\; \lambda'(A_p)+\lambda'(B_q)+
 \lambda'\Big ((k_A(p)\otimes_{k_R(n)}k_B(q))_{P(k_A(p)\otimes_{k_R(n)}k_B(q))}\Big)$$
and vanishing on any field (equivalently taking finite values on any field), 
 where $A$ and $B$ are $R$-algebras and $P\in\Spec(A\otimes_RB)$ with
 $p:=P\cap A$, $q:=P\cap B$ and $n:=P\cap R$, and either $A_p$ or $B_q$ is flat over $R_n$.

 We are going to see that $\Lambda=\Lambda'$.

\begin{thm}\label{2.3} Let $A$ and $B$ be two
$R$-algebras. Let $P\in$
$\Spec(A\otimes_RB)$ with $p:=P\cap A$, $q:=P\cap B$ and $n:=P\cap
R$. Assume that $A_p$ is flat over $R_n$. Then, for any $\lambda\in
\Lambda$,
$$\lambda((A\otimes_RB)_P) \;=\; \lambda\Big (\displaystyle {\frac {A_p}{nA_p}}\Big )+\lambda(B_q)+\lambda\Big
((k_A(p)\otimes_{k_R(n)}k_B(q))_{P(k_A(p)\otimes_{k_R(n)}k_B(q))}\Big
) \ .$$
\end{thm}

The proof of Theorem \ref{2.3} requires the following preparatory
lemma.

\begin{lemma}\label{2.4}
Let $R$ be a ring. Let $A$ and $B$ be $R$-algebras.

(1) Let $J$ be an ideal of $B$ such that $J\cap R:=m$ is a maximal
ideal of $R$. Then
$$\frac {A\otimes_RB}{A\otimes_RJ} \;\cong\; \frac A{mA}\otimes_k\frac
BJ\;\cong\; A\otimes_R\frac BJ$$ where $k$ denotes the field
$\displaystyle {\frac Rm}$.

(2) Let $P$ be a prime ideal of $A\otimes_RB$. Let $p:=P\cap A$,
$q:=P\cap B$ and $n:=P\cap R$. Then $$\frac
{(A\otimes_{R}B)_P}{(A\otimes_{R}q)_P}\;\cong\;
\big(\frac{A_p}{nA_p}\otimes_{k_R(n)}k_B(q)\big)
_{P(\frac{A_p}{nA_p}\otimes_{k_R(n)}k_B(q))} \ .$$
\end{lemma}
\begin{proof}
(1) Since $A\otimes_RJ=mA\otimes_RB+A\otimes_RJ$ because
$mA\otimes_RB\subset A\otimes_RJ$, using the exact sequences
$$ mA \;\longrightarrow\; A \;\longrightarrow\; A/mA \;\longrightarrow\; 0$$
$$ J \;\longrightarrow\; B \;\longrightarrow\; B/J \;\longrightarrow\; 0$$
it follows from \cite[Chapter II, \S 3.6, Proposition 6]{Bo} that
$$\frac {A\otimes_RB}{A\otimes_RJ} \;\cong\;
\frac {A\otimes_RB}{mA\otimes_RB + A\otimes_RJ}\;\cong\;
 \frac A{mA}\otimes_R\frac BJ \;\cong\;
 \big(\frac A{mA}\otimes_Rk\big)\otimes_k\frac BJ\;\cong\;
  \frac A{mA}\otimes_k\frac BJ \ .$$
  
(2) First, note that
$(A\otimes_RB)_P\cong(A_p\otimes_{R_n}B_q)_{P(A_p\otimes_{R_n}B_q)}$.
Then
$$\frac{(A\otimes_{R}B)_P}{(A\otimes_{R}q)_P}\;\cong\;
\frac{(A_p\otimes_{R_n}B_q)_{P(A_p\otimes_{R_n}B_q)}}{(A_p\otimes_{R_n}qB_q)_{P(A_p\otimes_{R_n}B_q)}}
\;\cong\;
\Big (\frac {A_p\otimes_{R_n}B_q}{A_p\otimes_{R_n}qB_q}\Big )_{\frac
{P(A_p\otimes_{R_n}B_q)}{A_p\otimes_{R_n}qB_q}} \ .$$
Now, as $qB_q\cap R_n=nR_n$ is maximal in $R_n$, we get the desired isomorphism
applying (1).
\end{proof}

\begin{proof}[ Proof of Theorem \ref{2.3}]
First, observe that $(A\otimes_RB)_P\cong
(A_p\otimes_{R_n}B_q)_{P(A_p\otimes_{R_n}B_q)}$. Then, without loss
of generality, we may assume that $R$, $A$ and $B$ are local rings with
maximal ideals $n$, $p$ and $q$, respectively, such that $A$ is flat over
$R$, $P(A\otimes_RB)_P\cap A=p$, $P(A\otimes_RB)_P\cap B=q$ and
$P(A\otimes_RB)_P\cap R=p\cap R=q\cap R=n$. As the ring homomorphism
$R\rightarrow A$ is flat, we get $B\rightarrow A\otimes_RB$ is a
flat homomorphism which induces the flat homomorphism of rings
$B\rightarrow (A\otimes_RB)_P$. Let $\lambda\in \Lambda$. Then, by hypothesis,
$$\begin{array}{lll}\lambda((A\otimes_RB)_P)&=&\lambda(B)+\lambda\Big (\displaystyle {\frac {(A\otimes_RB)_P}{
q(A\otimes_RB)_P}}\Big )\\
&&\\
&=&\lambda(B)+\lambda\Big (\Big (\displaystyle {\frac {A\otimes_RB}{
A\otimes_Rq}}\Big )_{\frac P{A\otimes_Rq}}\Big) \ .\end{array}$$ Hence,
by Lemma \ref{2.4}, as $q\cap R=n$ is maximal in $R$, we get
$$\lambda((A\otimes_RB)_P)\;=\;\lambda(B)+\lambda\Big (\Big
(\displaystyle {\frac A{nA}}\otimes_{k_R(n)}k_B(q)\Big )_{P(\frac
A{nA}\otimes_{k_R(n)}k_B(q))}\Big ) \ .$$
Since
$$ \displaystyle {\frac {A}{nA}} \;\longrightarrow\; \Big
(\displaystyle {\frac {A}{nA}}\otimes_{k_R(n)}k_B(q)\Big
)_{P(\frac{A_{p}}{nA_{p}}\otimes_{k_R(n)}k_B(q))}$$ is flat, we have

$$\lambda\Big (\Big (\displaystyle {\frac
A{nA}}\otimes_{k_R(n)}k_B(q)\Big )_{P(\frac
A{nA}\otimes_{k_R(n)}k_B(q))}\Big ) \;=\; \lambda\Big (\displaystyle
{\frac {A}{nA}}\Big )+\lambda\Big (\Big
(k_A(p)\otimes_{k_R(n)}k_B(q)\Big
)_{P(k_A(p)\otimes_{k_R(n)}k_B(q))}\Big ) \ .$$ It follows that
$$\lambda((A\otimes_RB)_P)\;=\;\lambda\Big (\displaystyle {\frac
{A}{nA}}\Big )+\lambda(B)+\lambda\Big (\Big
(k_A(p)\otimes_{k_R(n)}k_B(q)\Big
)_{P(k_A(p)\otimes_{k_R(n)}k_B(q))}\Big )$$ completing the proof.
\end{proof}

 \begin{corollary}\label{2.5}
 $\Lambda=\Lambda'$.
 \end{corollary}
 \begin{proof}
 If $\lambda\in\Lambda$, with the notation of Theorem \ref{2.3}, we have
 $$\lambda\big(\frac{A_p}{nA_p}\big)+\lambda(R_n)\;=\;\lambda(A_p) \ .$$
 Then $\lambda\in\Lambda'$, by Theorem \ref{2.3}.

 Now, if $\lambda'\in \Lambda'$, and $(S,m,k)\longrightarrow(T,n,l)$ is a flat local homomorphism, then
 $$\lambda'(S)+\lambda'\big(\frac{T}{mT}\big)\;=\;\lambda'(S)+\lambda'(T\otimes_Sk)=\lambda'(T)+\lambda'(k)+
 \lambda'(l\otimes_kk)=\lambda'(T)$$
 since $\lambda'$ vanishes on a field, so $\lambda'\in \Lambda$.
 \end{proof}

Our next result deals with the Krull dimension as well as the depth of
local tensor products of algebras over an arbitrary ring $R$.
The case of a base field was obtained in \cite[Proposition 2.3]{BK1}.

\begin{corollary}\label{2.6} Let $A$ and $B$ be $R$-algebras. Let
$P$ be a prime ideal of $A\otimes_RB$ with $n:=P\cap R$, $p:=P\cap
A$ and $q:=P\cap B$. Assume that $A_p$ is flat over $R_n$. Then

(1) $\dim((A\otimes_RB)_P)\;=\;\dim(A_p)+\dim(B_q)-\dim(R_n)+$

\hspace{4cm} $ \dim\Big ((k_A(p)\otimes_{k_R(n)}k_B(q))_{P(k_A(p)\otimes_{k_R(n)}k_B(q))}\Big )$ .

(2) $\depth((A\otimes_RB)_P)\;=\;\depth(A_p)+\depth(B_q)-\depth(R_n)+$

\hspace{4cm} $ \dim\Big
((k_A(p)\otimes_{k_R(n)}k_B(q))_{P(k_A(p)\otimes_{k_R(n)}k_B(q))}\Big
)$ .
\end{corollary}
\begin{proof} By Proposition \ref{2.2}, $\dim\in
\Lambda$ and $\depth\in \Lambda$. Then the result follows
from Corollary \ref{2.5} since the Krull dimension and the depth are
finite invariants when applied to Noetherian local rings, in other
terms, dim, depth: $\mathcal R\longrightarrow \mathbb{N}$.
\end{proof}

 \begin{remark}\label{2.7}
The Krull dimension of $k_A(p)\otimes_{k_R(n)}k_B(q)$ is the minimum
of the transcendence degrees of the field extensions
$k_A(p)|k_R(n)$ and $k_B(q)|k_R(n)$. This was proved in \cite[Quatri\'eme Partie, Err$_\text{VI}$.19, page 349]{EGA4}. See also \cite{Sh2} and \cite{Mu}.
\end{remark}

Recall that the ``codepth''
of a Noetherian local ring $A$
denoted by $\codepth(A):=\dim(A)-\depth(A)$ is an invariant
introduced by Grothendieck in \cite[Premi\`ere Partie, $0_{\text{IV}}$.16.4.9]{EGA4} to measure the
Cohen-Macaulayness defect of $A$. Next, we determine the
codepth of local tensor products of algebras over a ring.

\begin{corollary}\label{2.8}
Let $A$ and $B$ be $R$-algebras. Let
$P$ be a prime ideal of $A\otimes_RB$ with $n:=P\cap R$, $p:=P\cap
A$ and $q:=P\cap B$. Assume that $A_p$ is flat over $R_n$. Then
$$\begin{array}{lll}\codepth((A\otimes_RB)_P)&=&\codepth(A_p)+\codepth(B_q)-\codepth(R_n)\\
&=&\codepth\Big (\displaystyle {\frac {A_p}{nA_p}}\Big
)+\codepth(B_q) \ .\end{array}$$
\end{corollary}
\begin{proof}
The first equality follows from Corollary \ref{2.6} and the second from Proposition \ref{2.2}.
\end{proof}

We end this section by discussing the Cohen-Macaulayness of tensor
products of algebras over a ring $R$. We begin with the local case
and characterize when the localization of $A\otimes_RB$ is
Cohen-Macaulay.

\begin{corollary}\label{2.9} Let $A$ and $B$ be $R$-algebras. Let
$P$ be a prime ideal of $A\otimes_RB$ with $n:=P\cap R$, $p:=P\cap
A$ and $q:=P\cap B$. Assume that $A_p$ is flat over $R_n$. Then the
following assertions are equivalent:

(1) $(A\otimes_RB)_P$ is a Cohen-Macaulay ring;

(2) $\displaystyle {\frac {A_p}{nA_p}}$ and $B_q$ are Cohen-Macaulay
rings;

(3) $B_q$ is a Cohen-Macaulay ring and $\codepth(A_p)=\codepth(R_n)$.
\end{corollary}

\begin{proof} It is direct by the preceding corollary.\end{proof}

\begin{corollary}\label{2.10} Let $A$ and $B$ be $R$-algebras. Let
$P$ be a prime ideal of $A\otimes_RB$ with $n:=P\cap R$, $p:=P\cap
A$ and $q:=P\cap B$. Assume that $A_p$ and $B_q$ are flat over
$R_n$. Then the following assertions are equivalent:

(1) $(A\otimes_RB)_P$ is a Cohen-Macaulay ring;

(2) $A_p$ and $B_q$ are Cohen-Macaulay rings.\end{corollary}

\begin{proof} (2) $\Rightarrow$ (1) It is direct by Corollary \ref{2.9} as $R_n\longrightarrow A_p$
is a flat homomorphism.\\
(1) $\Rightarrow$ (2) By the above corollary $B_q$ is Cohen-Macaulay, and then so is $A_p$ by symmetry.
\end{proof}

\begin{corollary}\label{2.11}
Let $A$ and $B$ be $R$-algebras. Let
$P$ be a prime ideal of $A\otimes_RB$ with $n:=P\cap R$, $p:=P\cap
A$ and $q:=P\cap B$. Assume that $A_p$ is flat over $R_n$ and $R_n$
is a Cohen-Macaulay ring. Then the following assertions are
equivalent:

(1) $(A\otimes_RB)_P$ is a Cohen-Macaulay ring;

(2) $A_p$ and $B_q$ are Cohen-Macaulay rings.\end{corollary}

Next, we deal with the global case of Cohen-Macaulayness of tensor
products. The following corollaries generalize \cite[Theorem
2.1]{BK1} which characterizes the Cohen-Macaulayness of $A\otimes_kB$
in terms of the Cohen-Macaulayness of $A$ and $B$ in the setting of
a field $k$.

\begin{corollary}\label{2.12} Let $A$ and $B$ be $R$-algebras. Assume that $A$ is flat over $R$. Then the
following assertions are equivalent:

(1) $A\otimes_RB$ is a Cohen-Macaulay ring;

(2) $\displaystyle {\frac {A_p}{nA_p}}$ and $B_q$ are Cohen-Macaulay
rings for any prime ideals $p$ of $A$ and $q$ of $B$ such that
$p\cap R=q\cap R=n$.
\end{corollary}
\begin{proof} (1) $\Rightarrow$ (2) Assume that (1) holds. Let $p\in\Spec(A)$ and
$q\in\Spec(B)$ such that $p\cap R=q\cap R=:n$. Then by Proposition
\ref{2.1} there exists a prime ideal $P$ of $A\otimes_RB$ such that
$P\cap A=p$ and $P\cap B=q$. By (1), $A\otimes_RB$ is
Cohen-Macaulay, thus $(A\otimes_RB)_P$ is Cohen-Macaulay. As a consequence, by
Corollary \ref{2.9}, $B_q$ and $\displaystyle {\frac {A_p}{nA_p}}$ are
Cohen-Macaulay.\\
(2) $\Rightarrow$ (1) Assume that (2) holds. Let $P$ be a prime ideal
of $A\otimes_RB$ and let $n:=P\cap R$, $p:=P\cap A$, $q:=P\cap B$.
As $p\cap R=q\cap R=P\cap R=n$, we get by (2) that $\displaystyle
{\frac {A_p}{nA_p}}$ and $B_q$ are Cohen-Macaulay rings. Hence, by
Corollary \ref{2.9}, $(A\otimes_RB)_P$ is Cohen-Macaulay. It follows
that $A\otimes_RB$ is Cohen-Macaulay completing the proof.
\end{proof}

\begin{corollary}\label{2.13} Let $A$ and $B$ be $R$-algebras. Assume that $A$ and $B$ are flat over $R$. Then the
following assertions are equivalent:

(1) $A\otimes_RB$ is a Cohen-Macaulay ring;

(2) $A_p$ and $B_q$ are Cohen-Macaulay rings for any prime ideals $p$
of $A$ and $q$ of $B$ such that $p\cap R=q\cap R$.
\end{corollary}
\begin{proof}
It is similar to that of Corollary \ref{2.12} using Corollary \ref{2.10}.
\end{proof}

\begin{corollary}\label{2.14} Let $R$ be a Cohen-Macaulay ring. Let $A$ and $B$ be $R$-algebras. Assume that $A$ is flat over $R$. Then the
following assertions are equivalent:

(1) $A\otimes_RB$ is a Cohen-Macaulay ring;

(2) $A_p$ and $B_q$ are Cohen-Macaulay rings for any prime ideals $p$
of $A$ and $q$ of $B$ such that $p\cap R=q\cap R$.
\end{corollary}
\begin{proof} Again, the proof is similar to that of Corollary \ref{2.12}, but using
Corollary \ref{2.11}.
\end{proof}

\section{Gorenstein rings}

The aim of this section is to determine the injective dimension of
local tensor products of algebras over an arbitrary ring $R$. This
permits to generalize \cite[Theorem 6]{TY} on Gorensteiness of
tensor products of algebras over a field.

We begin by announcing the main theorem of this section. It gives
the injective dimension of local tensor products of algebras over a
ring $R$.

\begin{thm}\label{3.1} Let $A$ and $B$ be $R$-algebras. Let
$P$ be a prime ideal of $A\otimes_RB$ with $n:=P\cap R$, $p:=P\cap
A$ and $q:=P\cap B$. Assume that $A_p$ is flat over $R_n$.
Then
$$\idd_{(A\otimes_RB)_P}((A\otimes_RB)_P)\;=\;
\idd_{\frac {A_p}{nA_p}}\Big (\displaystyle {\frac {A_p}{nA_p}}\Big)+
\idd_{B_q}(B_q)+
\dim\Big((k_A(p)\otimes_{k_R(n)}k_B(q))_{P(k_A(p)\otimes_{k_R(n)}k_B(q))}\Big) \ .$$
\end{thm}

\begin{proof}
The proof is straightforward from Theorem \ref{2.3} since, by Proposition \ref{2.2},
$\idd\in \Lambda$, and using Bass's formula \cite[Lemma 3.3]{B} we have that
$$\idd_D(D)\;=\;\depth(D)\;=\;\dim(D) \ ,$$
as $D:=(k_A(p)\otimes_{k}k_B(q))_{P(k_A(p)\otimes_{k}k_B(q))}$ is a
Gorenstein local ring by \cite[Theorem 6]{TY}.

\end{proof}

The formula for the injective dimension of local tensor products of
algebras over a field $k$ is simpler as recorded next.

\begin{corollary}\label{3.2} Let $k$ be a field. Let $A$ and $B$ be $k$-algebras. Let
$P$ be a prime ideal of $A\otimes_kB$ with $p:=P\cap A$ and
$q:=P\cap B$. Then
$$\idd_{(A\otimes_kB)_P}((A\otimes_kB)_P)\;=\;
\idd_{A_p}(A_p)+\idd_{B_q}(B_q)+\dim\Big
((k_A(p)\otimes_{k}k_B(q))_{P(k_A(p)\otimes_{k}k_B(q))}\Big ) \ .$$
\end{corollary}

The remaining of this section is devoted to discuss the Gorensteiness of tensor products
of algebras over a ring $R$ in the local case as well as in the
global one. In particular, we generalize \cite[Theorem 6]{TY}
which proves that a tensor product $A\otimes_kB$ of algebras $A$ and
$B$ over a field $k$ is Gorenstein if and only if $A$ and $B$ are
so. The following corollaries are direct consequences of Theorem
\ref{3.1} so that the proofs are omitted. It suffices to recall that
$(k_A(p)\otimes_{k_R(n)}k_B(q))_{P(k_A(p)\otimes_{k_R(n)}k_B(q))}$
is a Gorenstein local ring.

\begin{corollary}\label{3.3} Let $A$ and $B$ be $R$-algebras. Let $P\in\Spec(A\otimes_RB)$ and let $n:=P\cap R$, $p:=P\cap A$ and $q:=P\cap
B$. Assume that $R_n\longrightarrow A_p$ is a flat ring homomorphism.
Then the following assertions are equivalent:

(1) $(A\otimes_RB)_P$ is a Gorenstein ring;

(2) $\displaystyle {\frac {A_p}{nA_p}}$ and $B_q$ are Gorenstein
rings.
\end{corollary}

\begin{corollary}\label{3.4} Let $A$ and $B$ be $R$-algebras. Let $P\in\Spec(A\otimes_RB)$ and let $n:=P\cap R$, $p:=P\cap A$ and $q:=P\cap
B$. Assume that $R_n\longrightarrow A_p$ is a flat ring homomorphism and
that $R_n$ is a Gorenstein ring. Then the following assertions are
equivalent:

(1) $(A\otimes_RB)_P$ is a Gorenstein ring;

(2) $A_p$ and $B_q$ are Gorenstein rings.
\end{corollary}

\begin{corollary}\label{3.5}
Let $A$ and $B$ be $R$-algebras. Assume that $R\longrightarrow A$ is a flat ring homomorphism.
Then the following assertions are equivalent:

(1) $A\otimes_RB$ is a Gorenstein ring;

(2) $\displaystyle {\frac {A_p}{nA_p}}$ and $B_q$ are Gorenstein
rings for any prime ideals $p$ of $A$ and $q$ of $B$ such that
$p\cap R=q\cap R$.
\end{corollary}

\begin{corollary}\label{3.6}
Let $R$ be a Gorenstein ring. Let $A$ and $B$ be $R$-algebras. Assume that $A$ and $B$ are flat over $R$.
Then the following assertions are equivalent:

(1) $A\otimes_RB$ is a Gorenstein ring;

(2) $A_p$ and $B_q$ are Gorenstein rings for any prime ideals $p$ of
$A$ and $q$ of $B$ such that $p\cap R=q\cap R$.
\end{corollary}

\section{Type of local rings}

Let $(R,m,k)$ be a local ring, where $m$ is the maximal ideal of $R$
and $k$ its residue field. Recall that the type of an $R$-module
$M$, denoted by $r_R(M)$, is an invariant which refines the
information given by the depth of $M$, namely $r_R(M):=\dim_k(\Ext_R^t(k,M))$
where $t=\depth_R(M)$. In this section, we express the type of the tensor product of algebras over a ring in terms of 
the type of its components.

We begin 
by recording the following theorem which examines the behavior of the type under flat ring homomorphisms.

\begin{thm} \label{4.1}
Let $(R,m,k)\longrightarrow (S,n,l)$ be a flat ring homomorphism of
local Noetherian rings. Then 
$$r_S(S)\;=\;r_R(R)\cdot r_{\frac S{mS}}\Big
(\displaystyle {\frac S{mS}}\Big ) \ .$$
\end{thm}
\begin{proof}
See \cite[Theorem]{FT}.
\end{proof}

The following result is a direct consequence of this theorem.

\begin{corollary}\label{4.2}
The composed function ln $\circ r$ of the the Neperian logarithm ln and the type $r$ is an element of $\Lambda$.
\end{corollary}

Recall that, by \cite[Theorem 3.2.10]{BH}, a local ring $A$ is
Gorenstein if and only if $A$ is a Cohen-Macaulay ring of type $1$
\footnote{
Note that if the type is defined as $r_R(R):=\dim_k(\Ext_R^t(k,R))$ where $t=\dim R$ instead of $t=\depth_R(R)$, then a local ring is
Gorenstein if and only if it is of type $1$ \cite{R}.
}.
From this perspective, the type of local rings might be regarded as
an invariant which measures the defect of Gorensteiness of local
rings. The following result examines the behavior of the property of
being ``of type $1$'' under flat ring homomorphisms.

\begin{corollary}\label{4.3}
Let $(R,m,k)\longrightarrow (S,n,l)$ be a flat ring homomorphism of
local Noetherian rings. Then $S$ is of type $1$ if and only if so
are $R$ and $\displaystyle {\frac S{mS}}$.
\end{corollary}

\begin{proof}
The proof is straightforward from Theorem \ref{4.1}.
\end{proof}

Next, we announce the main theorem of this section. It computes the
type of a local tensor product of algebras over a ring $R$ in terms
of the types of their components. Given a ring $A$ and when no
confusion is likely, we denote by $r(A)$ the type $r_A(A)$ of $A$
over itself.

\begin{thm}\label{4.4} Let $A$ and $B$ be $R$-algebras. Let $P$ be a prime
ideal of $A\otimes_RB$ and let $n:=P\cap R$, $p:=P\cap A$ and
$q:=P\cap B$. Assume that $A_p$ is flat over $R_n$. Then
$$r((A\otimes_RB)_P) \;=\;
r\Big(\dfrac{A_p}{nA_p}\Big)\cdot r(B_q) \;=\;
\dfrac{r(A_p)\cdot r(B_q)}{r(R_n)} \ .$$
\end{thm}
\begin{proof}
The proof follows from Theorem \ref{4.1}, Theorem \ref{2.3} and Corollary
\ref{4.2}, noting that $(k_A(p)\otimes_{k_R(n)}k_B(q))_{Pk_A(p)\otimes_{k_R(n)}k_B(q)}$
is a Gorenstein local ring and thus of type $1$.
\end{proof}

Now, we record when a local tensor product of algebras over
a ring behaves nicely with respect to Gorensteiness, namely when it
is of type $1$.

\begin{corollary}\label{4.5}
Let $A$ and $B$ be $R$-algebras. Let $P$ be a prime
ideal of $A\otimes_RB$ and let $n:=P\cap R$, $p:=P\cap A$ and
$q:=P\cap B$. Assume that $A_p$ is flat over $R_n$. Then
$(A\otimes_RB)_P$ is of type $1$ if and only if $\displaystyle
{\frac {A_p}{nA_p}}$ and $B_q$ are of type $1$.
\end{corollary}
\begin{proof}
It is clear from Theorem \ref{4.4}.
\end{proof}

Our last results of this section determine the type of local tensor
products of algebras over a field $k$ and examine when these
constructions are of type $1$.

\begin{corollary}\label{4.6}
Let $k$ be a field and let $A$ and $B$ be $k$-algebras. Let $P$ be a prime
ideal of $A\otimes_kB$. Let $p:=P\cap A$ and $q:=P\cap B$. Then
$$r((A\otimes_kB)_P)\;=\;r(A_p)\cdot r(B_q) \ .$$
\end{corollary}

\begin{corollary}\label{4.7}
 Let $k$ be a field and let $A$ and $B$ be $k$-algebras. Let $P$ be a prime
ideal of $A\otimes_kB$. Let $p:=P\cap A$ and $q:=P\cap B$. Then
$(A\otimes_kB)_P$ is of type $1$ if and only if $A_p$ and $B_q$ are
of type $1$.
\end{corollary}

\section{Introduction to Andr\'{e}-Quillen homology}

This section represents a short introduction to Andr\'e-Quillen
homology. If $A\longrightarrow B$ is a ring homomorphism and $M$ is a
$B$-module, then $H_n(A,B,M)$ denotes the Andr\'{e}-Quillen homology
$B$-module for each $n\geq 0$ \cite{An}. In the following, we list the basic
properties of these homology modules used in the remaining sections.

\begin{properties}
\textbf{(i)} (Functoriality) They are natural in all three variables
$A$, $B$ and $M$.\\* \textbf{(ii)} If $A\longrightarrow B$ is an isomorphism,
then $H_n(A,B,M)=0$ for all $n\geq 0$. If $B\cong \displaystyle {\frac AI}$, then
$H_0(A,B,M)=0$ and $$H_1(A,B,M)\;\cong\; \displaystyle {\frac I{I^2}}\otimes_{B}M \ .$$ In
particular, if $(A,m,k)$ is a Noetherian local ring, $\dim_F
H_1(A,k,F)=\embdim(A)$ for any field extension $F|k$.\\*
\textbf{(iii)} (Base change) Let $A\longrightarrow B$, $A\longrightarrow C$ be ring
homomorphisms and $t$ an integer such that $\Tor_n^A(B,C)=0$ for all
$0<n<t$. Then for any $B\otimes_AC$-module $M$, the canonical
homomorphisms
$$H_n(A,B,M)\,\longrightarrow\, H_n(C,B\otimes_AC,M)$$
are isomorphisms for all $n<t$. \\* \textbf{(iv)} Let $A\longrightarrow B\longrightarrow C$
be ring homomorphisms and $M$ a flat $C$-module. Then
$$H_n(A,B,M)\;\cong\; H_n(A,B,C)\otimes_CM $$ for all $n$.\\*
\textbf{(v)} (Jacobi-Zariski exact sequence) Let $A\longrightarrow B\longrightarrow C$ be
ring homomorphisms and $M$ a $C$-module. Then there exist a natural
exact sequence
\begin{align*}
  \cdots \; &\longrightarrow \; H_n(A,B,M)    \; \longrightarrow \;     H_n(A,C,M)   \; \longrightarrow \;    H_n(B,C,M)   \; \longrightarrow \;  \\
  &\longrightarrow \; H_{n-1}(A,B,M) \; \longrightarrow \;   \cdots\\
  \cdots \;  &\longrightarrow \; H_0(A,B,M)    \; \longrightarrow \;     H_0(A,C,M)   \; \longrightarrow \;    H_0(B,C,M)   \; \longrightarrow \;  0
\end{align*}
\textbf{(vi)} (Tensor product) Let $B$ and $C$ be $A$-algebras and
$t$ an integer such that $\Tor_n^A(B,C)=0$ for all $0<n<t$. Then for
any $B\otimes_AC$-algebra $D$ and any $D$-module $M$ we have a
natural exact sequence
\begin{align*}
& H_t(B\otimes_AC,D,M)    \; \longrightarrow \; \\
\longrightarrow\; H_{t-1}(A,D,M)   \; \longrightarrow \;    H_{t-1}(B,D,M)\oplus H_{t-1}(C,D,M)   \; \longrightarrow \; & H_{t-1}(B\otimes_AC,D,M) \;\longrightarrow \cdots \\
\cdots \; \longrightarrow \; & H_0(B\otimes_AC,D,M) \;
\longrightarrow \;  0
\end{align*}
\textbf{(vii)} (Field extensions) If $L| K$ is a field extension and
$M$ an $L$-module, then $H_n(K,L,M)=0$ for all $n\geq 2$. Moreover,
$L| K$ is separable if and only if $H_1(K,L,M)=0$ for some
(equivalently any) $L$-module $M\neq 0$. Therefore if $A\longrightarrow K\longrightarrow L$
are ring homomorphisms with $K$ and $L$ fields, then by (v)
$H_n(A,K,M)\longrightarrow H_n(A,L,M)$ is an isomorphism for all $n\geq 2$ and
injective for $n=1$. If $L\mid K$ is separable, it is also an
isomorphism for $n=1$.\\* \textbf{(viii)} (Localization) Let
$f\colon  A \longrightarrow B$ be a ring homomorphism and let $S\subset A$,
$T\subset B$ be multiplicatively closed subsets such that
$f(S)\subset T$. Let $M$ be a  $B$-module, then
\begin{align*}
T^{-1}H_n(A,B,M) \;&\cong\; H_n(A,B,T^{-1}M)\;\cong\; H_n(A,T^{-1}B,T^{-1}M)\\
\;&\cong\; H_n(S^{-1}A,T^{-1}B,T^{-1}M)
\end{align*}
for all $n\geq 0$. \\* \textbf{(ix)} (Finiteness) If $B$ an
$A$-algebra of finite type and $M$ a $B$-module of finite type, then
$H_n(A,B,M)$ is a $B$-module of finite type for all $n\geq0$.\\*
\textbf{(x)} (Vanishing) A Noetherian local ring $(A,m,k)$ is regular
if and only if $H_2(A,k,k)=0$ if and only if $H_n(A,k,k)=0$ for all
$n\geq 2$. If $A$ is a ring and $(B,n,l)$ a local $A$-algebra, then
$B$ is formally smooth for the $n$-adic topology
\cite[Premi\`ere Partie, $0_{\text{IV}}$.19.3.1]{EGA4} if and only if
$H_1(A,B,l)=0$.\\* \textbf{(xi)} (Flat extensions) If $(A,m,k)\longrightarrow
(B,n,l)$ is a local flat homomorphism of local rings, then the
homomorphism $H_2(A,k,l)\longrightarrow  H_2(B,l,l)$ is injective.\\*
\textbf{(xii)} (Completions) Let $(A,m ,k)$ and $(B,n ,l)$ be two
local rings with their completions $\hat{A}$ and $\hat{B}$. Then
$$H_n(A,k,k)\;\cong\; H_n(\hat{A},k,k) \qquad\text{ and }\qquad  H_n(A,B,l)\;\cong\; H_n(\hat{A},\hat{B},l)$$ for all $n$.

\end{properties}
\begin{proof}
(i) See \cite[3.15]{An}.\\*
(ii) See \cite[4.43, 4.60, 6.1]{An}.\\*
(iii) See \cite[9.31]{An}. \\*
(iv) See \cite[3.20]{An}.\\*
(v) See \cite[5.1]{An}.\\*
(vi) Same proof than in \cite[5.21]{An}, but using \cite[9.31]{An}
instead of \cite[4.54]{An}.\\*
(vii) See \cite[7.4, 7.13]{An}.\\*
(viii) The first isomorphism is \cite[4.59]{An} and the last two
follow from \cite[5.27]{An}.\\* 
(ix) See \cite[4.55]{An}.\\* 
(x) \cite[6.26]{An} gives the first claim. The second follows from
\cite[16.17, 3.20, 3.21]{An}.\\*
(xi) See \cite[Remarks (1.4)]{Av}.\\* 
(xii) The first isomorphism is
\cite[10.18]{An} and the second one is an immediate consequence of
the first (a reference for this isomorphism is \cite[Lemma 1]{FR}).
\end{proof}

\section{Complete intersection}

The aim of this section is to evaluate the complete
intersection defect of the tensor product of two algebras over a ring.

Let $(S,m,l)$ be a local ring. The \textit{complete intersection
defect} of $S$ \cite{KK} is defined as
$$\dd(S)\;:=\;\varepsilon_2(S)-\embdim(S)+\dim(S),$$ where $\varepsilon_2(S)$
denotes $\dim_lH_2(S,l,l)$, which is finite by Property (ix). It is remarkable that, by
\cite[15.12]{An} (see also \cite[2.5.1]{MR}), $\varepsilon_2(S)$
coincides with the dimension of the first Koszul homology module, denoted by $\dim
H_1(m)$, associated to a minimal set of generators of the ideal $m$.
Also, notice that $\dd(S)\geq 0$, and $\dd(S)=0$ if and only if $S$
is complete intersection \cite[p. 349]{KK} (or in terms of
Andr\'e-Quillen homology see \cite[4.3.4, 4.3.5]{MR}).

Recall that, in \cite[Theorem
2]{M}, it is proved the following theorem on
complete intersection of tensor products: \emph{Let $A$ and $B$ two
$R$-algebras such that for each
maximal ideal $P$ of $A\otimes_RB$, with $n=P\cap R$, $p=P\cap A$
and $q=P\cap B$, $A_p$ or $B_q$ is flat over $R_n$. If $A$ and $B$
are complete intersection then so is $A\otimes_RB$}.
Our next theorem generalizes this result by evaluating the
complete intersection defect of such constructions.

\begin{thm}\label{6.1}
Let $A$ and $B$ be $R$-algebras. Let $P\in\Spec(A\otimes_RB)$ with $p:=P\cap A$,
$q:=P\cap B$ and $n:=P\cap R$. Assume that $A_p$ is flat over $R_n$.
Then $$d((A\otimes_RB)_P)\;=\;d(A_p)+d(B_q)-d(R_n) \ .$$\end{thm}

\begin{proof} It is immediate from Proposition \ref{2.2} and Corollary \ref{2.5} as\\
$(k_A(p)\otimes_{k_R(n)}k_B(q))_{P(k_A(p)\otimes_{k_R(n)}k_B(q))}$
is complete intersection.
\end{proof}

Next, we isolate the case of tensor products of algebras over a field.

\begin{corollary} Let $k$ be a field and let $A$ and $B$ be $k$-algebras. Let $P\in\Spec(A\otimes_kB)$ with $p:=P\cap A$ and
$q:=P\cap B$. 
Then $$\dd((A\otimes_kB)_P)\;=\;\dd(A_p)+\dd(B_q) \ .$$\end{corollary}

The next corollaries characterizes when the tensor product of algebras over a ring is complete intersection.

\begin{corollary} Let $A$ and $B$ be $R$-algebras. Let $P\in\Spec(A\otimes_RB)$ with $p:=P\cap A$,
$q:=P\cap B$ and $n:=P\cap R$. Assume that $A_p$ is flat over $R_n$.
Then the following assertions are equivalent:
  
(1) $(A\otimes_RB)_P$ is complete intersection;

(2) $\dfrac{A_p}{nA_p}$ and $B_q$ are complete intersections;

(3) $B_q$ is complete intersection and $\dd(A_p)=\dd(R_n)$.
\end{corollary}

\begin{proof} It suffices to note, by Theorem \ref{6.1}, that $\dd((A\otimes_RB)_P)=\dd\big(\dfrac{A_p}{nA_p}\big)+\dd(B_q)$.
\end{proof}

\begin{corollary} Let $A$ and $B$ be $R$-algebras. Let $P\in\Spec(A\otimes_RB)$ with $p:=P\cap A$,
$q:=P\cap B$ and $n:=P\cap R$. Assume that $A_p$ and $B_q$ are flat over $R_n$.
Then the following assertions are equivalent:
  
(1) $(A\otimes_RB)_P$ is complete intersection;

(2) $A_p$ and $B_q$ are complete intersections.\end{corollary}

\begin{proof} It is direct as if either $A_p$ or $B_q$ is complete intersection, then so is $R_n$.\end{proof}

Note that one may characterize when the tensor product $A\otimes_RB$ is complete intersection by using Proposition \ref{2.1} 
as done in Section 2 for Cohen-Macaulayness.

Let $(R,n , k_R(n))$ be a Noetherian local ring and $\hat{R}$ its completion. It is well known (see \cite{K}) that for any representation of $\hat{R}$ 
as $\hat{R}=\displaystyle {\frac SI}$, where S is a regular local ring and $I$ an ideal of $R$, we have
$$\dd(R) = \mu(I) - dim(S) + dim(R) \ .$$
Here,  $\mu(I)=\dim H_1(S,\hat{R},k_R(n))$ stands for the minimum number of generators of $I$ (Property (ii)).
$R$ is said to be an \textit{almost complete intersection} ring if $\dd(R) = 1$.

The following two corollaries characterize when the tensor product of two $R$-algebras is almost complete intersection.

\begin{corollary}\label{6.5}
Let $A$ and $B$ be $R$-algebras. Let $P\in\Spec(A\otimes_RB)$ with $p:=P\cap A$,
$q:=P\cap B$ and $n:=P\cap R$. Assume that $A_p$ is flat over
$R_n$. Then $(A\otimes_RB)_P$ is almost complete intersection if and only if one of the following assertions holds:

(1) $\displaystyle {\frac {A_p}{nA_p}}$ is complete intersection and $B_q$ is almost complete intersection.

(2) $\displaystyle {\frac {A_p}{nA_p}}$ is almost complete intersection and $B_q$ is complete intersection.
\end{corollary}

\begin{proof} It is clear as by Theorem \ref{6.1}, $\dd((A\otimes_RB)_P)=\dd\big(\displaystyle {\frac {A_p}{nA_p}}\big)+d(B_q)$.\end{proof}

Note that under the conditions of Corollary \ref{6.5} if $R_n$, $A_p$ and $B_q$ are almost complete intersection, then so is 
$(A\otimes_RB)_P$.

\begin{corollary}
Let $A$ and $B$ be $R$-algebras. Let $P\in\Spec(A\otimes_RB)$ with $p:=P\cap A$,
$q:=P\cap B$ and $n:=P\cap R$. Assume that $A_p$ and $B_q$ are flat over
$R_n$. Then $(A\otimes_RB)_P$ is almost complete intersection if and only if one of the following assertions holds:

(1) $R_n$, $A_p$ and $B_q$ are almost complete intersection rings.

(2) $R_n$ and $A_p$ are complete intersection and $B_q$ is almost complete intersection.

(3) $R_n$ and $B_q$ are complete intersection and $A_p$ is almost complete intersection.
\end{corollary} 

\begin{proof} First, as by Theorem \ref{6.1}, $\dd((A\otimes_RB)_P)=\dd\big(\displaystyle {\frac {A_p}{nA_p}}\big)+\dd(B_q)=\dd(A_p)+
\dd\big(\displaystyle {\frac {B_q}{nB_q}}\big)$, note that if $(A\otimes_RB)_P$ is almost complete intersection,
then $\dd(R_n)\leq 1$, $\dd(A_p)\leq 1$ and $\dd(B_q)\leq 1$. 
If $R_n$ is complete intersection, then, by Corollary \ref{6.5}, $(A\otimes_RB)_P$ is almost complete intersection if and only if
either $A_p$ is almost complete intersection and $B_q$ is complete intersection or $A_p$ is complete intersection and $B_q$ is almost
complete intersection, as desired. Now, assume that $R_n$ is almost complete intersection. As $\dd(A_p)=\dd(R_n)+
\dd\big(\displaystyle {\frac {A_p}{nA_p}}\big)\geq 1$, if $\dd(A_p)\leq 1$, then $\dd(A_p)=1$, that is, $A_p$ is almost complete intersection. Similarly, if
$\dd(B_q)\leq 1$, then $B_q$ is almost complete intersection. It follows that $(A\otimes_{R}B)_P$ is almost complete intersection 
if and only if $A_p$ and $B_q$ are so. This completes the proof of the corollary.

\end{proof}

\section{Regular rings}

We are mainly concerned in this section with measuring the defect of
regularity of the tensor product of algebras over a ring.

For a local ring $S$, we denote by
$\codim S:= \embdim S-\dim S$ the codimension of $S$, so that $S$ is
regular if and only if $\codim S=0$. Thereby, from the very definition of complete intersection defect of $S$,
note that
$$\varepsilon_2(S)=\dd(S)+\codim(S)$$ where $\varepsilon_2(S)=\dim_lH_2(S,l,l)$.
Therefore $\varepsilon_2(S)$, as well as $\codim(S)$, is a measure of the defect of regularity of $S$ (it also follows from Property (x)).

\begin{lemma}\label{7.1}
Let $A$ and $B$ be two $R$-algebras. Let $P$ be a prime ideal
of $A\otimes_RB$ with $n:=P\cap R$, $p:=P\cap A$ and $q:=P\cap B$.
Assume that $A_p$ is a flat $R_n$-module.
Then, the following statements are equivalent:

(1) $\varepsilon_2((A\otimes_RB)_{P})\;=\; \varepsilon_2(A_p) + \varepsilon_2(B_q)- \varepsilon_2 (R_n) $;

(2) $\codim((A\otimes_RB)_{P})\;=\; \codim(A_p) + \codim(B_q)
-\codim(R_n) $.
\end{lemma}

\begin{proof} It suffices to note that, by Theorem \ref{6.1},
$$\dd((A\otimes_RB)_{P})\;=\; \dd(A_p)+ \dd(B_q)-\dd(R_n)$$ and
thus
$$\begin{array}{lll}\varepsilon_2((A\otimes_RB)_P)-\varepsilon_2(A_p)-\varepsilon_2(B_q)+\varepsilon_2(R_n)&=&
\codim((A\otimes_RB)_P)-\codim(A_p)\\
&& \quad - \; \codim(B_q)+\codim(R_n) \ .\end{array}$$
\end{proof}

Recently, in \cite[Theorem
1]{M}, it is proved the following theorem on
regularity of tensor products: \emph{Let $R$ be a Noetherian ring and
let $A$ and $B$ be two Noetherian $R$-algebras such that
$A\otimes_RB$ is a Noetherian ring. Assume that for each maximal
ideal $P$ of $A\otimes_RB$ with $n=P\cap R$, $q=P\cap A$ and
$q=P\cap B$ at least one of the three local $R_n$-algebras $A_p$,
$B_q$ or $k_{A\otimes_RB}(P)$ is formally smooth for the topology of
its maximal ideal. If $A$ and $B$ are regular rings, then so is
$A\otimes_RB$}. Our next results generalize this
theorem by measuring the defect of regularity of these tensor
products under the formal smoothness of one the components over the
base ring. 

We begin by dealing with the codimension of rings issued from formally smooth homomorphisms.

\begin{thm}\label{7.2}
Let $A$ and $B$ be two $R$-algebras. Let $P$ be a prime ideal of
$A\otimes_RB$ with $n:=P\cap R$, $p:=P\cap A$ and $q:=P\cap B$. If
$R_n\longrightarrow A_p$ is formally smooth, then
$$\codim((A\otimes_RB)_{P})\;=\; \codim(A_p) + \codim(B_q) -\codim(R_n) \ .$$
\end{thm}
\begin{proof} 
Assume that $R_n\longrightarrow A_p$ is formally smooth. In light of Lemma
\ref{7.1}, it suffices to prove that
$\varepsilon_2((A\otimes_RB)_{P})\;=\; \varepsilon_2(A_p) +
\varepsilon_2(B_q) -\varepsilon_2(R_n) $. Observe that, by
\cite[Premi\`ere Partie, $0_{\text{IV}}$.19.7.1]{EGA4}, $A_p$ is a flat $R_n$-module.
Now, using Property (vi), consider the exact sequence
(where $E$ is the residue field of $(A\otimes_RB)_P$)
\begin{align*}
&H_{2}(R_{n},E,E)   \; \overset{\alpha_2}{\longrightarrow} \;    H_{2}(A_{p},E,E)\oplus H_{2}(B_{q},E,E)   \; \longrightarrow \;  H_{2}(A_{p}\otimes_{R_{n}}B_{q},E,E) \;\longrightarrow \\
\longrightarrow \; & H_{1}(R_{n},E,E)
\;\overset{\alpha_1}{\longrightarrow}\;  H_{1}(A_{p},E,E)\oplus
H_{1}(B_{q},E,E) \ .
\end{align*}
Let us show that $\alpha_1$ and $\alpha_2$ are injective. Note that
$\alpha_2$ is injective by Property (xi).
Also, Property (v) yields the following exact sequence
$$H_1(R_n,A_p,E)=0 \;\overset{\beta}{\longrightarrow}\; H_1(R_n,E,E) \;\longrightarrow\; H_1(A_p,E,E)$$
by Property (x) as $R_n\longrightarrow A_p$ is formally smooth. Therefore
$\alpha_1$ is injective. Hence the following sequence is exact
\begin{align*}
&0{\longrightarrow} \;H_2(R_n,E,E){\longrightarrow} \;
H_{2}(A_{p},E,E)\oplus H_{2}(B_{q},E,E) \; \longrightarrow \;
H_{2}(A_{p}\otimes_{R_{n}}B_{q},E,E) \longrightarrow 0 \
.\end{align*} As $H_2(A_p\otimes_{R_n}B_q,E,E)\cong
H_2((A\otimes_RB)_P,E,E)$ by Property (viii), and $\dim_EH_2(A_p,E,E)=\varepsilon_2(A_p)$
by Property (vii) (and similarly for $R_n$ and $B_q$), it follows that
$$\varepsilon_2 ((A\otimes_RB)_P)\;=\;\varepsilon_2 (A_p)+\varepsilon_2 (B_q)-\varepsilon_2 (R_n)$$ completing the proof of the theorem.
\end{proof}

\begin{corollary}\label{7.3} Let $(A,m,k)\longrightarrow (B,n,l)$ be a formally smooth homomorphism of local rings. Then
 
(1) $\varepsilon_2(B)=\varepsilon_2(A)+\varepsilon_2\big(\displaystyle {\frac B{mB}}\big)$.

(2) $\codim(B)=\codim(A)+\codim\big(\displaystyle {\frac B{mB}}\big)$.
\end{corollary}

\begin{proof} 
Part (2) follows from Theorem \ref{7.2} and Lemma \ref{7.1}, since
\begin{align*}
\codim\big(\frac{B}{mB}\big) =
\codim(B\otimes_Ak) = \codim B -\codim A + \codim k = \codim B - \codim A
\end{align*}
and (1) follows from (2) as above.
\end{proof}

The following corollaries discuss the regularity property as well as the embedding
dimension of the tensor products of algebras over 
a ring under formal smoothness hypothesis.

\begin{corollary}
Let $A$ and $B$ be two $R$-algebras. Let $P$ be a prime ideal of
$A\otimes_RB$ with $n:=P\cap R$, $p:=P\cap A$ and $q:=P\cap B$. Assume that
$R_n\longrightarrow A_p$ is formally smooth. Consider the following assertions:

(1) $A_p$ and $B_q$ are regular.

(2) $(A\otimes_RB)_P$ is regular;

(3) $\displaystyle {\frac {A_p}{nA_p}}$ and $B_q$ are regular;

(4) $B_q$ is regular and $\codim(A_p)=\codim(R_n)$.

Then (1) $\Rightarrow$ (2) $\Leftrightarrow$ (3) $\Leftrightarrow$ (4).
\end{corollary}

\begin{proof} It is direct from Theorem \ref{7.2} as by Corollary \ref{7.3}, $\codim(A_p)=\codim(R_n)+\codim\big(\displaystyle {\frac {A_p}{nA_p}}\big)$.
\end{proof}

\begin{corollary}
Let $A$ and $B$ be two $R$-algebras. Let $P$ be a prime ideal of
$A\otimes_RB$ with $n:=P\cap R$, $p:=P\cap A$ and $q:=P\cap B$. If
$R_n\longrightarrow A_p$ is formally smooth, then
\begin{align*}
\embdim((A\otimes_RB)_P)\;&=\;\embdim (A_p)+\embdim (B_q) -\embdim (R_{n}) \\
 &\qquad+ \; \dim \big((k_A(p)\otimes_{k_R(n)} k_B(q))_{P(k_A(p)\otimes_{k_R(n)} k_B(q))}\big) \ .
\end{align*}
\end{corollary}
\begin{proof}
It follows from Corollary \ref{2.6} and Theorem \ref{7.2}.
\end{proof}

Next, we announce our second main theorem of this section.

\begin{thm}
Let $A$ and $B$ be $R$-algebras. Let $P$ be a prime ideal of
$A\otimes_RB$ with $p:=P\cap A$ and $q:=P\cap B$. Assume that
$R_n\longrightarrow k_A(p)$ is formally smooth, that is, $R_n$ is a field and
$k_A(p)$ is separable over $R_n$. Then

(1) $\codim((A\otimes_RB)_{P})\;=\; \codim(A_p) + \codim(B_q).$

(2) $\embdim((A\otimes_RB)_{P})\;=\; \embdim(A_p) +
\embdim(B_q)$

\hspace{4cm} $+\;\dim\Big
((k_A(p)\otimes_{R_n}k_B(q))_{P(k_A(p)\otimes_{R_n}k_B(q))}\Big)$ .
\end{thm}
\begin{proof}
The proof of (1) is the same as that of Theorem \ref{7.2}, taking in mind that the
homomorphism $\beta$ in that proof factorizes as
$$H_1(R_n,A_p,E) \;\longrightarrow\; H_1(R_n,k_A(p),E) \;\longrightarrow\; H_1(R_n,E,E)$$
and $H_1(R_n,k_A(p),E)=0$ by Property (vii). Then (2) follows from (1) and Corollary \ref{2.6}.
\end{proof}

\begin{corollary} Let $A$ and $B$ be $R$-algebras. Let $P$ be a prime ideal of
$A\otimes_RB$ with $p:=P\cap A$ and $q:=P\cap B$. Assume that
$R_n\longrightarrow k_A(p)$ is formally smooth (i.e., $R_n$ is a field and
$k_A(p)$ is separable over $R_n$). Then the following assertions are equivalent:

(1) $(A\otimes_RB)_P$ is regular;

(2) $A_p$ and $B_q$ are regular.\end{corollary}


\end{document}